\documentclass{article}
\usepackage{amssymb}
\usepackage{amsfonts}

\newtheorem{theorem}{Theorem}

\newtheorem{lemma}[theorem]{Lemma}

\newtheorem{remark}[theorem]{Remark}

\newenvironment{proof}[1][Proof]{\noindent\textbf{#1.} }{\ \rule{0.5em}{0.5em}}
\begin{document}

\title{Some comparisons between the Variational rationality, Habitual
domain, and DMCS approaches }
\date{September, 28, 2014}
\author{G. C. Bento \and A. Soubeyran}

\maketitle


The \textquotedblleft Habitual domain"\textbf{\ (}HD) approach and the
\textquotedblleft Variational rationality" (VR) approach belong to the same
strongly interdisciplinary and very dispersed area of research: human
stability and change dynamics (see Soubeyran, 2009, 2010, for an extended
survey), including physiological, physical, psychological and strategic
aspects, in Psychology, Economics, Management Sciences, Decision theory,
Game theory, Sociology, Philosophy, Artificial Intelligence,\ldots. These
two approaches are complementary. They have strong similarities and strong
differences. They focus attention on both similar and different stay and
change problems, using different concepts and different mathematical
tools. When they use similar concepts (a lot), they often have different
meaning. We can compare them with respect to the problems and topics they
consider, the behavioral principles they use, the concepts they modelize,
the mathematical tools they use, and their results.

\section{Problems and topics}

\paragraph{The \textquotedblleft Habitual domain" (HD) theory.}

For a survey see Yu, Chen (2010). Among other points, this approach focus
attention on:

\begin{itemize}
\item[A)] behavioral habitual domain principles (Yu, 1991, Yu, Chen, 2010);

\item[B)] habit formation and the dynamic of activation levels and attention
efforts, following a succession of periods (Chan, Yu, 1985);

\item[C)] problem solving, modelized as a given competency set expansion
(CSE) problem within a given period (see Shi, Yu, 1999). It represents the
most developped part of the HD theory, which includes a lot of related
papers using different mathematical tools to find the optimal expansion set
(see among others, Yu, Zhang, 1990, Li, Yu, 1994, Shi, Yu 1996, Tzeng and
alii, 1998, Li, Chiang, Yu, 2000,\ldots);

\item[D)] innovation dynamics (ID), consider that competency sets are dynamic
and can change over time. It examines a succession of competency set, expansion problems which refer to decision making with changeable spaces (DMCS) problems. Innovation dynamics examine also cover-discover (CD) problems (Larbani,Yu, 2012). A covering problem is a competency set
expansion (CSE) problem. It refers to \textquotedblleft how to transform a
given competence set CS* into a set that contains a targeted competence set
CS\textquotedblright. Discovering refers to the following problem:
\textquotedblleft Given a competence set, what is the best way to make use
of it to solve unsolved problems or to create value? The process under the problem solving or value creation, involves discovering. A discovering
process can be defined as identifying how to use available tangible and
intangible skills, competences, and resources to solve an unsolved problem
or to produce new ideas, concepts, products, or services that satisfy some
newly-emerging needs of people. ....Discovering contributes to reducing the
charge level or relieving the pain of some targeted people" (Larbani,Yu,
2012);

\item[E)] Decision making with changeable spaces (DMCS) games with changing
minds (Yu, Larbani, 2009, Larbani, Yu, 2009, 2009b, 2011, 2012).
\end{itemize}

\paragraph{The \textquotedblleft Variational rationality" (VR) approach}

For an extended presentation of this model, see Soubeyran (2009, 2010). 
This approach focuses attention on a single very general problem: the famous self regulation problem, which involves a lot of many multi-disciplinary aspects in different disciplines, employing different terminologies. Among other more specific topics, it considers:

\begin{itemize}
\item[A')] behavioral stability and change principles;

\item[B')] habit formation and break (HFB) problems at the individual level,
as well as routine formation and break (RFB) problems at the organizational
level;

\item[C')] exploration-exploitation (EEL) learning dynamics;

\item[D')] adaptive self regulation (SR) problems (goal setting, goal
striving, goal revision and goal pursuit) which belong to the general class
of decision making with changeable spaces, moving goals and variable
preferences (DMCSMG) problems. Such self regulation processes modelize the
interrelated dynamics of motivational (desiring and willing), cognitive
(perception and knowledge acquisition) and emotional aspects of human
behavior. Bounded rational satisficing processes represent an important
application;

\item[E')] variational games (VG) with self regulating agents (Attouch et al., 2007, Attouch et al., 2008, Flores-Baz\'{a}n, Luc, Soubeyran,
2012, Flam et al., 2012, Cruz Neto et al., 2013);

\item[F')] application of self regulation problems to variational analysis.
Variational analysis and most of the related algorithms: variational
principles (Luc, Soubeyran, 2013) exact and inexact alternating algorithms
(Attouch et al., \ 2007, Attouch et al., 2008, Attouch et al., 2010),
proximal algorithms with Bregman and quasi distances (Cruz Neto et al.,
2010, Moreno et al., 2012), exact and inexact proximal algorithms on
manifolds (Cruz Neto et al., 2013), multiobjective proximal algorithms
(Bento, Cruz Neto, Soubeyran, 2013), local search inexact proximal
algorithms (Attouch, Soubeyran, 2010), inexact descent methods (Bento and
alii, 2013), equilibrium problems (Bento et al., 2013), and dual
equilibrium problems (Moreno et al., 2012), variational inequalities
(Attouch, Soubeyran, 2006, Luc et al., 2010), trust region methods
(Villacorta et al., 2013), Tabu search algorithms (Martinez-Legaz,
Soubeyran, 2007), sequential decision making (Martinez-Legaz, Soubeyran,
2013) etc, use the same variational principles as VR. These algorithms
may be seen as reduced forms of the VR self regulation model.
\end{itemize}

\section{ Behavioral principles}

\paragraph{HD behavioral principles}

HD theory is based on eight behavioral principles \ (Yu, Chen, 2010). Four
hypotheses capture the basic workings of the brain: Circuit pattern
hypothesis H1, Unlimited capacity hypothesis H2, Efficient restructuring
hypothesis H3 and Analogy/Association hypothesis H4. Four other hypotheses
summarize how our mind works: Goal setting and State evaluation hypothesis
H5 (a basic function of our mind), Charge structure and attention allocation
hypothesis H6 (how we allocate our attention to various events), Discharge
hypothesis H7 (a least resistance principle that humans use to release their
charges) and the Information internal and external inputs hypothesis H8.

\paragraph{VR behavioral principles}

VR approach is based on, at least, nine behavioral principles (Soubeyran,
2009, 2010):

\begin{itemize}
\item[K1)] agents are bounded rational (Simon, 1955). They do not optimize,
except for simple problems. As soon as a problem is complex, it satisfice
within a consideration set. They consider, in each step, only a limited subset
of alternatives related to the behavioral chain: \textquotedblleft means and
capabilities ---$>$ actions ----$>$ performances ----$>$ goals ----$>$
desires". This consideration set changes from one period to the next;

\item[K 2)] human activities follow a succession of temporary stays and
changes;

\item[K3)] in each period, the agent is satisfied, or remains unsatisfied,
relative to different domains of it's life;

\item[K4)] the agent problem, in each step, is  to choose one of two alternatives: to temporary
stay or to change (\textquotedblleft should I stay, should I go?");

\item[K5)] the agent balances, in each step, between advantages and
inconvenients to change or to stay, and more generally, motivation and
resistance to change or to stay. They consider worthwhile temporary stays or
changes;

\item[K6)] an agent is engaged in a \textquotedblleft stop and go" stays and
changes course pursuit, between setting, each period, the same old desired
ends and/or new ones, as well as finding related feasible means to reach or
approach each of them;

\item[K7)] an agent is partially able to self regulate this possibly
interrupted course pursuit. He can set goals, strive for them, and revise
them (having reached some goals, to reset the same goals is a possibility);

\item[K8)] experience matters and partially determines the current
behavioral chain. Then, almost all things, including preferences, can
change. There is a course pursuit between changing preferences, actions,
capabilities and beliefs. A current action is chosen on the basis of the
current preference. This changes the current preference which, in turn,
changes the choice of the future action which, in turn\ldots;

\item[K9)] Goal pursuit stops when the agent reaches a behavioral trap,
which is worthwhile reaching, starting from the initial position, and where
he prefers to stay than to move again.
\end{itemize}

\paragraph{Comparisons}

Circuit pattern hypothesis H1 refers to our resistance to change concept
(repetitions of thoughts, ideas and actions reinforce circuits, making them
difficult to abandon). Unlimited capacity hypothesis H2 states that
\textquotedblleft practically, every normal brain has the capacity to encode
and store all thoughts, concepts and messages that one intends to".
Efficient Restructuring hypothesis H3 states that \textquotedblleft encoded thoughts, concepts and messages H1 are organized and stored
systematically as data bases for efficient retrieving. Furthermore, according to the dictation of attention they are continuously restructured so that relevant ones may be  retrieved efficiently in order to release charges".
Analogy/Association hypothesis H4 supposes that \textquotedblleft the
perception of new events, subjects or ideas can be learned primarily by
analogy and/or association with what is already known. When faced with a new
event, subject or idea, the brain first investigates its features and
attributes in order to establish a relationship with what is already known;
by analogy and/or association. Once the right relationship has been
established, the whole of the past knowledge (preexisting memory structure)
is automatically brought to bear on the interpretation and understanding of
the new event, subject or idea". The above hypotheses can be used to derive bounded
rationality (Simon, 1955) and to our resistance to change
concept (see K5), where new knowledge and past knowledge are not mixed
immediately. For example, as a lot of experiences on habit formation and
break have shown in Psychology, many attitudes and beliefs temporary resist
to change.

The goal setting and state evaluation hypothesis H5 supposes that
\textquotedblleft each one of us has a set of goal functions and for each
goal function we have an ideal state or equilibrium point to reach and
maintain (goal setting). Continuously, consciously or
subconsciously we monitor, where we are, relative to the ideal state or equilibrium
point (state evaluation). Goal setting and state evaluation are dynamic,
interactive, and are subject to physiological forces, self-suggestion,
external information forces, current data bank (memory) and information
processing capacity". The VR approach supposes (see K6) variable ideal
states (moving aspirations and desires) and an adaptive course pursuit
between where we are (uncomfortable moving status quo) and where we want to
be (moving aspirations, or desirable ends).

The charge structures and attention allocation hypothesis H6 supposes that
\textquotedblleft each event is related to a set of goal functions. When
there is an unfavorable deviation of the perceived value from the ideal,
each goal function will produce various levels of charge. The totality of
the charges by all goal functions is called the charge structure and it may
change dynamically. At any point in time, our attention will be paid to the
event which has the most influence on our charge structure". Discharge hypothesis H7 supposes that \textquotedblleft to release charges, we tend to select the action which yields the lowest remaining charge (the remaining
charge is the resistance to the total discharge), known as the least resistance principle". The VR approach agrees with these two hypothesis H6 and H7. Indeed, this is  very similar to the famous \textquotedblleft discrepancy reduction" principle in Psychology. They represent the first
side of a self regulation process (see K6). The other side is a
\textquotedblleft discrepancy production" process (goal setting, goal
revision; goal pursuit).

The information Input hypothesis H8 supposes that \textquotedblleft humans
have innate needs to gather external information. Unless attention is paid,
external information inputs may not be processed". The VR approach agrees with this, which can be related to the general concept of consideration sets (see K1).

\section{Concepts, variables and parameters}

\paragraph{HD concepts.}

To save space, let us list these concepts with respect to only three
different HD problems:

\begin{itemize}
\item[i)] Stabilization of an habitual domain problem. Following an infinite
sequence of periods (a transition, in the parlance of the VR approach), the
agent considers, in each period $t$, his potential domain at time $t$, his
actual domain at time $t$, his activation probability at time $t$, his
reachable domain at time $t$, and his attention and activation levels at
time $t$;

\item[ii)] Competency set expansion problem. Staying within a given period of time, defined as a single change in the VR approach (the case of a transition, where the agent follows an infinite sequence of periods unexamined), the agent considers: his given competency set as well as his given acquired skill set, both true or perceived, a given table of costs
needed for acquiring a given skill directly from another given skill, his
cost of acquiring a new skill, a chosen optimal expansion process, decision traps, etc;

\item[iii)] Decision making in changeable spaces (DMCS) problems. They can
be represented by a time dependent list including, at each time $t$: i) a list
of changing or changeable decision elements (a subset of alternatives, a
subsets of criteria, an outcome measured in term of the criteria, and a
preference) and ii) a list of changing or changeable decision environmental
facets (information inputs, an habitual domain, a subset of other involved
decision making agents, and a subset of unknowns). All the elements of this
list can be changed with time.
\end{itemize}

Innovation problems, like successions of competency set expansion problems
and cover-discover problems, as well as DMCS game problems are very
important problems which, however, will not be examined here since our paper
considers only an agent.\textbf{\ }

\paragraph{VR concepts.}

They are relative to the self-regulation problem. More precisely, given
a transition (an infinite sequence of periods), the agent follows a
succession of single temporary stays or changes. Given his current
experience which changes with time, he considers, at each period, using his
current changeable consideration set, his current changeable aspiration gap. That is defined as the current discrepancy between where he is, the status quo,
and where he wants to be (representing his current changeable aspiration level), his current changeable goals, which helps to fill a portion of the current
aspiration gap, as well as, to approach his current aspiration level, a
future chosen action to be done, repeating the old action, or doing a new
action, the regeneration of old capabilities to be able to repeat the old
action, the deletion of old capabilities and the acquisition of new
capabilities and means to be able to do the new action, his expected new
performance and payoff generated by this action, his expected costs to be
able to stay and his expected costs to be able to change, his expected
advantages to change, inconvenient to change, motivation to change,
resistance to change. The distinction can be made between an ex ante
perceived and an ex post realized concept, variable or parameter, etc.. All
this elements can change. Then, worthwhile
changes and variational traps are the keys variational rationality concepts.

The same concepts can be used to examine the functioning of variational
games (for a reduced form, see inertial games, Attouch, Redont, Soubeyran,
2007, Flores-Baz\'{a}n, Luc, Soubeyran, 2012). To save space we do not
examine them in this paper which focus on a single agent.

\paragraph{Comparisons.}

\begin{itemize}
\item[1)] In the parlance of the VR approach, stability and change dynamics
consider two dynamics: an intra period dynamic (a single temporary stay or
change made of a succession of elementary stay or change operations) and an
inter periods dynamic (a succession of periods, named a transition). The
Chan, Yu (1985) paper on stable habitual domains examines transitions, while
a competency set expansion process (Shi,Yu, 1999) refers to an inter period
dynamic (a single change);

\item[2)] HD theory examines decision making with changeable spaces (DMCS)
problems, while the VR approach considers decision making with changeable
spaces and moving goals (DMCSMG) problems. However (DMCS) problems can
choose goals as well. VR self regulation processes focus attention on i)
goal setting and goal revision (hence changing aspiration levels and goals)
and ii) goal striving (discrepancy or charge reduction);

\item[3)] Main HD choice variables are skills, competency sets, and
activation propensities. Main VR choice variables are capabilities, actions
and goals where actions refer to paths of elementary operations.

\item[4)] VR and HD payoffs are experience dependent in their most general
formulations (however, the resolution of the competency set expansion
problem, being mathematically so complex, seems to require the opposite.
See, for example, Shi,Yu, 1996, theorem 5.1, where the expected return
function ER depends only on the elements acquired from Tr$\setminus $%
Sk\ldots );

\item[5)] The two concepts of competency set and capabilities, while having
some similarities, differ. In the HD theory, a competency set is a 
collection of resources and skills. For the VR approach a capability is a
path of operations including script and timing, related means
(physiological, physical, cognitive, motivational and emotional, tangible
and intangible ingredients, downstream and upstream tools and machines used
to perform these operations, following the given script); However, both
approaches consider, as alternative formulations, subsets and paths to
modelize competency sets and capabilities.

\item[6)] The HD definition of asymmetric costs for acquiring a new skill
directly from a given skill (see, for example Shi, Yu, 1996), is a
particular and reduced form of the VR definition of  change costs to be able to
change, defined as the cost to acquire, directly or indirectly, the capability
to do a new action, starting from having the capability to repeat an old
action. They include direct and indirect costs to delete some old elementary
capabilities, which should not be used anymore and pollute\ldots, costs to
regenerate some other old capabilities, which will be reused, and costs
to acquire new elementary capabilities. HD costs to acquire a new skill from
an old one are direct costs. They seem to include only acquisitions costs
and they represent the infimum of costs to be able to change (see Soubeyran,
2009);

\item[7)] The VR resistance to change concept (see Soubeyran, 2009, 2010)
differs from the HD resistance to change concept (Larbani,Yu, 2012). For the
HD theory, the discharge hypothesis supposes that \textquotedblleft to
release charges, we tend to select the action, which yields the lowest
remaining charge (the remaining charge is the resistance to the total
discharge); this is called the least resistance principle" (Larbani,Yu,
2012). In the VR theory resistance to change is the disutility of
inconvenients to change capabilities.

\item[8)] In the HD approach, motivation to change is modelized in term of
charges and discharges. The VR concept closest to this  will be the
utility-disutility of charges and discharges (as tensions). The VR
motivation-resistance to change balance may be compared to the
excitation-inhibition functions (Chan, Yu,1985);

\item[9)] The definition of traps differs. The HD theory (Larbani,Yu, 2012)
says that a decision maker is in a decision trap at time $t$, if his
competence set is trapped in some area and cannot expand to fully cover the
targeted competence set.  The resolution of challenging problems generally
involve covering and discovering. Discovering requires a
target to cover, while the covering process requires discovering when it
falls in a decision trap". The VR approach defines a variational trap with
respect to an initial situation (action, a doing, or a state, some having or
being). This type of situation is worthwhile to reach, starting from this
initial situation. However, being there, not worthwhile to leave. Optima,
equilibria, decision traps, habits, routines, rules and norms represent
specific cases. For game situations, win-win outcomes are examples of
variational traps.

\item[10)] The VR approach does not modelize the expansion of the competency
sets. On the contrary, Larbani,Yu (2012) give three HD toolboxes which show
how to expand and enrich the actual and reachable domains and look into the
depth of potential domains:

\begin{itemize}
\item[a)] the seven empowerment operators;

\item[b)] eight methods for expanding the habitual domain M8. Learning
actively, M9. Projecting from a higher position, M10. Active association,
M11. Changing the relevant parameters, M12. Changing the environment, M13.
Brainstorming, M14. Retreating in order to advance, M15. Praying or
meditating (Larbani, Yu, 2012);

\item[c)] nine principles of deep knowledge: M16 the deep and down
principle, M17 the alternating principle, M18, the contrasting and
complementing principle, M19, the revolving and cycling principle, M20. the
inner connection principle, M21 the changing and transforming principle,
M22 the contradiction principle, M23 the cracking and ripping principle,
M24 the void principle (Larbani, Yu, 2012).
\end{itemize}
\end{itemize}

\section{Rationality of a single agent.}

\paragraph{Rationality and the HD theory.}

In this paper we consider an isolated agent. Thus, the comparisons between
HD and VR game situations will be examined elsewhere. HD theory examines
three different problems and proposes three different behavioral models for
a single agent, who can be fully or boundedly rational, depending of the model.

\begin{itemize}
\item[i)] The stabilization of an habitual domain problem. This situation is
modelized by a stabilization of activation propensities model, which
represents a reduced form of the stabilization of an habitual domain problem
(Chan,Yu, 1985). This model is a differential equation, a variant of the
famous global pattern formation model (Cohen, Grossberg,1983). The authors
examine its convergence (weak and strong global stability). In this case an
agent does not optimize. He is boundedly rational;

\item[ii)] The competency set expansion problem. In this case the main focus
is on an optimal \textquotedblleft Problem solving" approach. More precisely,
an agent has a given problem E to solve. To succeed to solve his problem,
he must own a collection of skills defined as the competency set, Tr(E) (true
or perceived), related to the full resolution of the problem. The agent
starts the resolution equipped with a given competency set, the acquired
skill set Sk (true or perceived), defined as the collection of skills he owns
at the beginning, before starting the resolution. The problem is to find an
optimal path of expansion of his competency set, from the initial position
Sk to the final position Tr(E), representing a fixed given goal. A lot of
different algorithms have been used to find the optimal solution for
intermediate and compound skills and asymmetric cost functions (tree
expansion processes, deduction graphs, spanning trees,\ldots ; see Li,
Chiang, Yu, 2000). In this case the agent is fully (substantively) rational;
The opposite case of a bounded rational agent, using a satisficing
process to solve a competency set expansion problem, remains an 
interesting open problem in this area of research.

\item[iii)] Decision making and optimization in changeable spaces (DMOCS)
problems. In a new setting, Larbani, Yu (2012) used some dedicated
optimization methods and suggest to search for other new optimizing methods
to solve them. Cover-discover problems belong to this class of new
optimization problems. However, Larbani, Yu (2012) states that "the operator
Min in the models (6)-(8) should be understood in the sense of satisfaction,
not in the sense of absolute minimum". They notice (Larbani,Yu, 2012, p 742)
that optimization must be understood in term of reducing the charge level of
the decision maker to a satisfactory or acceptable level. This is in
accordance with the satisficing principle (Simon, 1955).
\end{itemize}

\paragraph{Bounded rationality and the VR approach.}

VR theory proposes a unified model for human behavior, focusing on
worthwhile temporary stays and changes, variational traps, and self
regulation processes (goal setting, goal striving, goal revision, goal
pursuit processes). It is well adapted to complex, changing and high stake
decision making problems (see Kunreuther et al., 2002), where agents
cannot be fully rational. In a complex and changing world, full
optimization, at each step, is too costly and even not economizing, because
situations (spaces of feasible means and capabilities, actions,
performances, payoffs, intermediate and final goals, desires and
aspirations) change during each step. Hence, an optimal solution at time $t$ may be
irrelevant at time $t+1$. An agent tries to reach, at each step, a moving
satisficing level (not a fixed one). In each step, he considers worthwhile
temporary stays or changes, which include, as special cases, adaptive,
satisficing, local, approximate, inexact solutions (and optimal solutions as
limit cases).

\section{Behavioral stability issues: \textquotedblleft how habits and
routines form and break".}

\paragraph{Different mathematical tools for stability issues.}

The HD theory uses, at least, three main mathematical tools to examine
stability issues (convergence to a stable and desirable final situation) and
innovation problems (reaching a targeted competency set, starting from a
given initial one): i) the dynamics of pattern formation (Grossberg, 1973,
1978, 1980, Cohen, Grossberg, 1983) as the main tool to examine the dynamic
of activation propensities (Chan, Yu, 1985), ii) mathematical programming
and different graph methods, to study competency set formation and innovation
problems (see, among many other papers, Shi,Yu, 1996), iii) Markov chains, to
examine the convergence of second order DMCS games to desirable and stable
issues (Yu, Larbani, 2009, Larbani,Yu, 2009, 2011, 2012). In contrast, the
VR approach (Soubeyran, 2009, 2010) starts from variational rationality
principles in Behavioral Sciences and offers, as immediate applications, a
lot of famous mathematical principles of variational analysis (Ekeland
theorem, Bronsted Lemma, and other equivalent principles,\ldots; see
Flores-Baz\'{a}n et al., 2012, Luc, Soubeyran, 2013 ). In turn, it uses a
lot of well known variational algorithms (proximal algorithms, descent
methods, variational inequalities, trust region methods, equilibrium
problems,\ldots) in order to help to refine the VR approach relative to stability and
innovative issues, for isolated and interacting agents (VR games).

While strongly related to the VR approach, competency set expansion problems
do not refer to stability issues, but to innovative issues, while the
dynamics of pattern formation (activation propensities) and second order
DMCS games main focus, are on stability issues. To save space, and since
our inexact proximal algorithm paper chooses, for an application, habit's
and routine's formation at the individual and organizational levels (a
benchmark stability issue), we will only compare how the HD and VR
approaches solve, in different ways, this very difficult problem of
\textquotedblleft habits and routines formation and break". The comparison of
the HD and VR approaches relative to DMCS problems, competency set expansion problems, innovation problems, and second order DMCS games, will be examined elsewhere. However a preliminary step is done, later, for the
comparison of second order DMCS games and VR games.

Let us compare how the HD theory, using the dynamics of pattern formation
(Grossberg, 1973, 1978, 1980, 1983) and the VR approach modelize and
explain how habit-and-routine form and break on two grounds: i) explain
convergence to a final issue, ii) explain why this final issue is desirable
and stable.

Notice that the pattern and the variational rationality methods both involve a balance principe. Worthwhile changes, balance motivation and resistance to change while changes in pattern allocations, balance excitation and inhibition inputs and signals.

\paragraph{HD theories of habit's and routine's formation and break.}

The HD intuition is the following (Chan, Yu,1985) \textquotedblleft \ldots
the existence of stable HD, based on a set of hypotheses is described.
Roughly, as each human being learns, his HD grows with time, but at a
decreasing rate, because the probability for an arrival idea to be new with
respect to HD, becomes smaller as HD gets larger. Thus, unless unexpected
extraordinary events arrive, HD will reach its stable state. If
extraordinary events do not arrive very often, habitual ways of thinking and
action will prevail most of the time. This observation is the main
motivation to use 'habitual'\ as the adjective. More formally, HD theory of
habit and routine formation considers the convergence of the allocations of
time and effort (activities propensities) to different activities up to a
final pattern (an habitual pattern of time and effort allocations), using a
variant of the famous pattern formation model (Grossberg, 1973, 1978, 1980,
1983). Notice that his model of progressive pattern formation does not
explain why this final allocation is desirable and stable. However, as said
before, in the contect of DMCS problems, win-win situations refer to
desirable and stable final outcomes (Yu, Larbani, 2009, Larbani,Yu, 2009,
2011, 2012).

\paragraph{ VR theories of habit's and routine's formation and break.}

The VR intuition is very different. Agents make worthwhile changes and stop
to change when there is no way to be able to consider and make a new
worthwhile change. The convergence may be in finite or infinite time (see
Bento, Soubeyran, 2014, Flores Bazan, Luc, Soubeyran, 2012; Bento, Cruz
Neto, Soares, Soubeyran, 2014). The model explains why, and under which
conditions, this final issue is desirable and stable. This is the case when
it is a variational trap. Moreover, the formalized VR theory of habit's and
routine's formation fits very well (see below) the main non formalized
experimental findings of the different theories of \textquotedblleft how
habits and routines form and break",  Psychology as well as in Management
Sciences, within a bounded rationality approach, and, to some degree, in
Economics, in a perfectly rationality context.

\textbf{Habit's formation in Psychology and Economics. }In Psychology habits
represent \textquotedblleft learned sequences of acts that have become
automatic responses to specific cues, and are functional in obtaining
certain goals or end states"; see Verplanken, Aarts (1999). For (Duhigg,
2012), an habit is an automatized action (mental or physical experience), a
more or less fixed way of thinking, willing, feeling and doing which follows
an automatized three steps pattern: a given trigger which activates the
action, a process (or script) that the action follows, and a reward (benefit
or gain). Hence habits are learned automatic behaviors. Repetition in a
similar recurrent context is a necessary condition for habits to develop.
Frequency of past behavior and context stability like internal cues (moods
and goals) and external cues ( partners and external goals) determine habit
strength. Habits represent a form of automaticity (Bargh,1994). They are
more or less conscious and intentional (wanted, i.e., the perception of
contexts is more or less goal directed, triggered by goals or by other cues).
They are learned in a progressive way. They can be difficult to control,
hard to form, because they follow a progressive learning process, and more
or less hard to break,\ given some weak or strong motivation and resistance
to change, as the vestige of past behavior. Then, they can resist to change.
Habits can be good (mentally efficient, saving on deliberation efforts).
Habits can also be bad \ (addictions, behavioral traps,\ldots).

In Economics agents are perfectly rational and habits are defined as stocks
of past experiences. A current habit is modelized as a stock of past
behaviors which determines the present preference of the agent with respect
to present consumption. In standard models of addictions, see (Becker,
1988), and habit formation, see (Abel, 1990, Carroll, 2000), preferences
have the given current numerical representation $U_{n}=U(c_{n},h_{n})$,
where the current state $h_{n}$ represents a stock of habits, $c_{n}$ stands
for current consumption, and $n$ indexes time. The habit persistence
hypothesis implies that instantaneous utility does not only depend on
current consumption, but also on a stock of habits, $h_{n}$.

\textbf{Routine's formation in Management Sciences.} In
this discipline\textbf{\ }routines are defined at the organizational level,
as collective patterns of interactions. An enormous literature considers
routines as organizational habits in the context of the stability and change
dynamics of organizations. The excellent survey is (Becker, 2004), which
lists the main points which characterize routines as: patterns of
interactions, collective activities, mindlessness vs effortful
accomplishments, processes (ways of doing, scripts), context dependent
(embeddedness and specificity), path dependent, and triggered by related
actors and external cues. Routines have several effects. They favor
coordination, control, truce and stability. They also economize on cognitive
resources, store knowledge, and reduce uncertainty.

\paragraph{Stability issues in VR and N person second order games}

Although our proximal algorithm paper considers only an agent and not a game
situation, the VR approach includes the examination of variational games
which follow, using reduced formulations, the VR list of nine principles
given above (Attouch et al., \ 2007, Attouch et al., 2008, Flores-Baz%
\'{a}n, Luc, Soubeyran, 2012, Flam et al., 2012, Cruz Neto et al.,
2013). Here, the nine VR principles will not be repeated. However (to save
space, this is left to the reader), to allow a possible more complete
comparison of HD and VR stability issues (possible convergence to some
stable final situation), let us summarize the content of N persons second
order games (Yu, Larbani, 2009, Larbani,Yu, 2009, 2011, 2012). They
incorporate human psychology in formulating games as people play them, using
the habitual domain theory and the Markov chain theory. A Markov chain
modelizes the evolution of the states of mind of players over time as
transition probabilities over them. States of mind determine the outcomes of
the so called two or $N$-person second-order game. The final issues are not
Nash equilibria, but focal mind profiles, which are desirable to reach and
globally stable solutions of the game, and win win profiles, which are focal
and absorbing profiles, while Nash equilibria are not. These games can
predict the average number of steps needed for a game to reach a focal or
win-win mind profile where both players declare victory. Given some
hypothesis, the "possibility theorem" states that it is always possible to
reach a win-win mind profile, restructuring (reframing) the data of the
game, the set of players, the payoff functions, and the set of strategies of
the initial game. In this reframing context, the HD information input
hypothesis H8 plays a major role. In such games, players are not fully
rational. They follow the rule of the HD theory. They refer to DMCS
problems, where the structure of the game can be restructured (it is
changeable), and information inputs help to reach a win win profile. To
summarize, in a conflict situation, second order players try to reach a
focal profile (which exists) and, once it is reached, they try to make it
stable, as a win win profile.

\section{Variational rationality, changeable payoffs and decision sets, and
the inexact proximal algorithm}

The Variational rationality approach considers the rationality of agents,
when many things change or are changeable, and other things must temporary
stay. Let us summarize our finding. First, the inexact proximal algorithm
(with relative resistance to change) is a reduced form of the variational
rationality model. Second, it is not a repeated optimizing problem in
changeable spaces as an exact proximal algorithm can be. In this section, we
will show that it represents a worthwhile to stay and change dynamic. Then,
it is an adaptive and satisficing course pursue problem, with changeable
payoffs, goals, and decision spaces.

\paragraph{Exact proximal algorithms as repeated optimization problems with
variable payoffs and changeable spaces}

Let us consider, first, exact proximal algorithms which are benchmark cases
of their inexact versions. They are not satisficing models of human
behaviors. However, they can be considered as repeated optimization problems
with variable payoffs and changeable decision sets (as worthwhile to change
sets, see the Lemma). The fact that they consider variable payoffs has been
shown above.

Let $X=\mathbb{R}^{n}$ be an action space, $f:X\rightarrow \mathbb{R}\cup
\{+\infty \}$ a proper, lower semicontinuous function bounded from below and
consider the following problems.

\begin{itemize}
\item \textbf{The} \textbf{fixed payoff and decision set optimization
problem (1)}: 
\[
\mbox{PROBLEM 1}:\qquad \inf \left\{ f(y),\;y\in X\right\} . 
\]%
This problem is the substantive (global) rationality minimization problem. 

\item \textbf{The \ fixed payoff and fixed decision set exact proximal
algorithm (2)}%
\[
\mbox{PROBLEM 2}:\qquad \inf \left\{ f(y)+\lambda _{k}\Gamma \left[
q(x^{k},y)\right] ,\;y\in X\right\} , 
\]%
where $x^{k}\in X$ and $\lambda _{k}\in \mathbb{R}_{++}$ are given for each $%
k\in \mathbb{N}$, $q:X\times X\rightarrow \mathbb{R}_{+}$ is a
quasi-distance and $\Gamma :\mathbb{R}\rightarrow \mathbb{R}$ represent the
relative resistance to change. This problem is the exact version of our
inexact proximal problem with relative resistance to change. It is a
repeated optimization problem.

\item \textbf{The variable payoffs and fixed decision set problem (2'):} 
\[
\mbox{PROBLEM 2'}:\qquad \inf \left\{ f_{E^{k}}(y)+\eta _{k}\Gamma \left[
q(x^{k},y)\right] ,\;y\in X\right\} . 
\]

where the variable payoff function is $%
f_{E^{k}}(y)=f_{E^{k}}(y)=v(E^{k})f(y) $

Equality $\lambda _{k}=\eta _{k}/v(E^{k})$ shows that this problem is
equivalent to Problem 2. It represents a variable and experience dependent
payoff with fixed decision set problem. It allows to define a course pursuit
problem with variable preferences.

\item \textbf{Variable payoffs and variable decision set problems. } 
\[
\mbox{PROBLEM 3}:\qquad \inf \left\{ f(y)+\lambda _{k}\Gamma \left[
q(x^{k},y)\right] ,\;y\in W_{\lambda _{k}}(x^{k})\right\} , 
\]%
where $W_{\lambda _{k}}(x^{k})=\left\{ y\in X:f(x^{k})-f(y)\geq \lambda
_{k}\Gamma \left[ q(x^{k},y)\right] \right\} $, \qquad k=0,1,\ldots .
\end{itemize}


\begin{lemma}
\emph{PROBLEM 1} and \emph{PROBLEM 2} are equivalent, i.e., 
\[
\emph{argmin}_{y\in X}\left\{ f(y)+\lambda _{k}\Gamma\left[q(x^{k},y)\right]%
\right\}=\emph{argmin}_{y\in W_{\lambda_{k}}(x^{k})}\left\{ f(y)+\lambda
_{k}\Gamma\left[q(x^{k},y)\right]\right\}. 
\]
\end{lemma}


\begin{proof}
Take $x_{2}^{k}\in \mbox{argmin}_{y\in X}\left\{ f(y)+\lambda _{k}\Gamma%
\left[q(x^{k},y)\right]\right\}$. Taking into account that $\Gamma\left[%
q(x^{k},x^{k})\right]=0$, from the definition of $W_{\lambda _{k}}(x^{k})$,
it follows immediately that $x_{2}^{k}\in W_{\lambda _{k}}(x^{k})$. Now,
take 
\[
x_{3}^{k}\in \arg \min \left\{ f(y)+\lambda_{k}\Gamma\left[q(x^{k},y)\right]%
, y\in W_{\lambda _{k}}(x^{k})\right\}. 
\]
It is easy to see that: 
\[
f(x_{3}^{k})+\lambda _{k}\Gamma\left[q(x^{k},x_{3})\right]^{k})\leq
f(x_{2}^{k})+\lambda _{k}\Gamma\left[q(x^{k},x_{2}^{k})\right]. 
\]
On the other hand, since $W_{\lambda _{k}}(x^{k})\subset X$ and $%
x_{3}^{k}\in W_{\lambda _{k}}(x^{k})$, definition of $x_{2}^{k}$ implies: 
\[
f(x_{3}^{k})+\lambda _{k}\Gamma\left[q(x^{k},x_{3}^{k})\right]\geq
f(x_{2}^{k})+\lambda _{k}\Gamma\left[q(x^{k},x_{2}^{k})\right], 
\]
%
%
%
%
%
%
and the result follows.
\end{proof}

\begin{remark}
This lemma show that our inexact proximal algorithm is a changeable payoff
and decision set process which belongs to the class of \textquotedblleft
Decision Making and Satisficing (not necessarily Optimizing) Problems in
Changeable Spaces", where the changeable payoff is $P_{\lambda
_{k}}(x^{k},y)=f(y)+\lambda _{k}\Gamma \left[ q(x^{k},y)\right] ,$ and the
changeable decision set is the current worthwhile to change set $W_{\lambda
_{k}}(x^{k})$.
\end{remark}

\paragraph{Inexact proximal algorithms as adaptive satisficing dynamics with
variable payoffs and changeable spaces}

If the agent follows an inexact proximal algorithm, he will choose to
perform, each step $k$, a worthwhile change $y\in W_{e_{k},\xi _{k}}(x^{k})$
where $\xi _{k}=\lambda _{k}\mu _{k}>0,$ 
\[
f(x^{k})-f(y)\geq \lambda _{k}\mu _{k}\Gamma \left[ q(x^{k},y)\right]
,\qquad 0<\mu _{k}\leq 1. 
\]%
This inexact proximal algorithm (with relative resistance to change)
introduces a lot of simplifications, as a reduced form of the variational
rationality model. Among others, it identifies actions $x\in X$ \ to
activities ${\huge x}=(x,\left[ x\right] )$ where $\left[ x\right] \in \left[
X\right] $ refers to a given capability to do an action $x$. It supposes a
strictly increasing and invertible pleasure (utility) function $U_{e}\left[
A_{e}\right] =A_{e}$, and defines a relative resistance to change function $%
\Gamma (I_{e})=U^{-1}\left[ D_{e}\left[ I_{e}\right] \right] $ where pain,
i.e, the disutility of inconvenients to change is $D_{e}\left[ I_{e}\right]
=I_{e}.$ It considers separable experience dependent unsatisfied need
functions $f_{E^{k}}(y)=v(E^{k})f(y)$ or separable payoff (profit) functions 
$g_{E^{k}}(y)=v(E^{k})g(y),$ where $E^{k}=(x^{1},x^{2},...,x^{k})$ is the
history of past actions, and the influence of experience is modelized via
the coefficient $v(E^{k})>0.$ It considers infimum costs to be able to
change $C(x,y)=\inf \left\{ C(\left[ x\right] ,\left[ y\right] ,\omega _{%
\left[ x\right] ,\left[ y\right] }),\omega _{\left[ x\right] ,\left[ y\right]
}\in \Omega (\left[ x\right] ,\left[ y\right] )\right\} ,$ among all finite
costs to be able to change $C(\left[ x\right] ,\left[ y\right] ,\omega _{%
\left[ x\right] ,\left[ y\right] })<+\infty $, following a path of change, $%
\omega _{\left[ x\right] ,\left[ y\right] }\in \Omega (\left[ x\right] ,%
\left[ y\right] $, from a given capability $\left[ x\right] $ to do an
action $x$ to a given capability $\left[ y\right] $ to do a new action $y$.
This inexact proximal algorithm allows to define worthwhile changes $%
x^{k}\curvearrowright $ $y$ as $g_{E^{k}}(y)-g_{E^{k}}(x^{k})\geq $ $\eta
_{k}\Gamma \left[ q(x^{k},y)\right] $, where $\eta _{k}>0$ is an adaptive
worthwhile to change satisficing ratio, which can be changed from period $k$
to period $k+1$. If we note $\lambda _{k}=\eta _{k}/v(E^{k}),$then, a
worthwhile change $x^{k}\curvearrowright $ $y$ is defined as $%
g(y)-g(x^{k})\geq $ $\lambda _{k}\Gamma \left[ q(x^{k},y)\right] $. In term
of separable experience dependent unsatisfied need functions $%
f_{E^{k}}(y)=v(E^{k})f(y),$ a worthwhile change $x^{k}\curvearrowright $ $%
y\in W_{e_{k},\xi _{k}}(x^{k})$ is such that $f(x^{k})-f(y)\geq $ $\lambda
_{k}\Gamma \left[ q(x^{k},y)\right] .$

The topic of our paper is not exact proximal algorithms, but inexact ones. \
Inexact proximal algorithms examined in this paper represent adaptive
satisficing dynamics (dealing with changeable satisficing levels), variable
and experience dependent preferences and changeable decision sets, which
belong to the class of decision making with changeable spaces and changeable
goals problems, noted DMCSCG problems. Let us show this important point,
which helps the comparison with the Habitual domain theory and DMCS decision
making problems with changeable spaces (Larbani, Yu, 2012). Our VR point of
view is the following. An inexact proximal algorithm is a specific instance
of a VR worthwhile to stay and change dynamics $x^{k+1}\in W_{e_{k},\xi
_{k+1}}(x^{k}),k\in N$ . This dynamic is both satisficing and considers
changeable goals and decision sets,

i) This dynamic is satisficing because the worthwhile to change condition
generalizes the Simon (1955) definition in a dynamical context, balancing
satisfactions to change with sacrifices to change. The moving goal is, each
period, the chosen worthwhile to change satisficing level $\xi _{k+1}>0.$

ii) This dynamic is adaptive, and considers changeable decision sets $%
W_{e_{k},\xi _{k+1}}(x^{k})$. Each period, the agent can chooses how much
changes must be worthwhile (the size of $\xi _{k+1}$) \ to accept to change.
Then, the worthwhile to change set is chosen, at each period.

\paragraph{Local inexact proximal algorithms}

For each $k\in \mathbb{N}$ fixed, let $X^{k}\subset X$, $x^{k}\in X$, $%
\lambda_{k}\in \mathbb{R}_{++}$ and consider the following problem: 
\[
\mbox{PROBLEM 4}:\qquad \inf\left\{ f(y)+\lambda _{k}\Gamma\left[q(x^{k},y)%
\right],\; y\in X^{k}\right\}. 
\]
This a variable preference proximal problem with a changeable feasibility
space $X^{k}\subset X$. In Attouch, Soubeyran (2010) it is a changing ball $%
X^{k}=B_{r}(x^{k})$ of constant radius $r>0$. We can define the indicator
function $I_{X^{k}}$ and consider the changeable payoffs and decision spaces
proximal problem, 
\[
\inf \left\{ f(y)+I_{X^{k}}(y)+\lambda _{k}\Gamma\left[q(x^{k},y)\right],\;
y\in X\right\}. 
\]
Bento et al. (2013) have also examined an exact local search
multiobjective proximal problem, where $X^{k}$ is a lower countour set of
the multi-objective function. 

Now, let us consider \textbf{``fixed payoff and variable decision set
problems"}, namely, 
\[
\mbox{PROBLEM 5}:\qquad \inf \left\{ f(y), y\in W_{\lambda
_{k}}(x^{k})\right\}. 
\]
%
Let us assume that $\{x_{5}^{k}\}$ is a generated sequence from the
iterative process 
\[
x_{5}^{k}\in \mbox{argmin}\left\{ f(y), y\in W_{\lambda_{k}}(x^{k})\right\}. 
\]
Since $x_{5}^{k}\in W_{\lambda _{k}}(x^{k})$, by definition, the agent
follows, each step, a worthwhile change.

\begin{lemma}
Let $\{x_{2}^{k}\}$ be a generated sequence from the iterative process 
\[
x_{2}^{k}\in \emph{argmin}_{y\in X}\left\{ f(y)+\lambda _{k}\Gamma\left[%
q(x^{k},y)\right]\right\}, 
\]
such that $\{f(x_{2}^{k})\}$ converges to $\underline{f}:=\inf \left\{
f(y),y\in X\right\}$. Then the sequence $\{x_{5}^{k}\}$ is a minimizing
sequence for \emph{PROBLEM~1}.
\end{lemma}

\begin{proof}
Take $k\in \mathbb{N}$ arbitrary. Note that $x^{k}_{2}\in W_{\lambda
_{k}}(x^{k})$ (see proof of \mbox{Lemma 2}). It is easy to see that $%
f(x_{5}^{k})\leq f(x_{2}^{k})$, i.e., each solution of \mbox{PROBLEM 2}
gives an estimation-majoration $f(x_{2}^{k})$ to the minimal payoff $%
f(x_{5}^{k})$ of \mbox{PROBLEM 5}. Now, from the definition of $\underline{f}
$, we have 
\begin{equation}  \label{eq:2013}
\underline{f}\leq f(x_{5}^{k})\leq f(x_{2}^{k}).
\end{equation}
Therefore, the desired result follows from (\ref{eq:2013}) together the fact
of $\{ x_{2}^{k}\}$ be a minimizing sequence for \mbox{PROBLEM 1}.
\end{proof}

\paragraph{Mathematical stability issues}

\begin{itemize}
\item[1)] \textbf{Stability of variational traps. }Let $\lambda _{k}=\eta
_{k}/v(E^{k})$ be a stage $k$ proximal ratio, where $\eta _{k}>0$ is a stage 
$k$ worthwhile to change ratio, and $v(E^{k})>0$ is the experience rate of
influence at stage $k.$ Let $\lambda _{\ast }=\eta _{\ast }/v(E^{\ast })$ be
the limit proximal ratio. It is easy to show that if $x^{\ast }\in X$ is a
variational trap, given the proximal ratio $\lambda _{\ast }$, then, $%
x^{\ast }\in X$ is also a variational trap, for any higher proximal ratio $%
\lambda >\lambda _{\ast }.$ This comes from the implication: $W_{\lambda
_{\ast }}(x^{\ast })=\left\{ x^{\ast }\right\} $ implies $W_{\lambda
}(x^{\ast })\subset W_{\lambda _{\ast }}(x^{\ast })$ for $\lambda >\lambda
_{\ast }.$ Then, $W_{\lambda }(x^{\ast })=W_{\lambda _{\ast }}(x^{\ast
})=\left\{ x^{\ast }\right\} $ for $\lambda >\lambda _{\ast }.$ A higher
proximal ratio $\lambda =\eta /v(E)>\lambda _{\ast }=\eta _{\ast }/v(E^{\ast
})$ means a higher worthwhile to change ratio $\eta >\eta _{\ast }$ or (and)
a lower experience rate of influence $v(E)<v(E^{\ast }).$ It can also mean
higher costs to be able to change parameters and, each period, a longuer
length of the exploitation phase, or a shorter length of the exploration
phase (Soubeyran, 2009, 2010). For example, stability issues with respect to
adaptive parameters like $\lambda _{k}\geq \lambda _{\infty }>0$, where $%
\lambda _{k}$ represents a changing preference, resistance to change, costs
to be able to change, and length of \ the exploitation period relative to
the exploration period, will be examined elsewhere.

\item[ 2)] \textbf{Stability of worthwhile to change trajectories. }For an
exact quadratic proximal problem, 
\[
\inf \left\{ f(y)+\lambda \Gamma \left[q(x,y)\right],y\in X\right\},\qquad
\Gamma\left[q(x,y)\right]=\left\Vert y-x\right\Vert ^{2}, 
\]
this is a well known result. The proximal mapping 
\[
X\ni y\longmapsto \varphi _{\lambda }(x):=\arg \min \left\{ f(y)+\lambda
\Gamma\left[q(x,y)\right],\; y\in X\right\}, 
\]
is well defined (recall that $f$ was assumed be bounded from below) and, if $%
f$ is convex, it is non expansive, i.e., 
\[
\left\vert \varphi _{\lambda }(x^{\prime })-\varphi _{\lambda
}(x)\right\vert \leq \left\Vert x^{\prime }-x\right\Vert,\qquad x,x^{\prime
}\in X, 
\]
see, for example, (Attouch, Soubeyran, 2010). For an inexact proximal
algorithm with a non quadratic regularization term and a non convex $f$,
this is an open difficult question.
\end{itemize}

\section{Habit's and Routine's Formation: Inexact Proximal Processes with
Weak Resistance to Change}

In this section (see the Arxiv version of our paper, Bento, Soubeyran,
2013), we detail how our inexact proximal algorithm modelizes habit/routine
formation and break, using the VR resistance to change modelization.

\paragraph{Comparing worthwhile changes and stays processes, inexact
proximal algorithms and habituation-routinization processes.}

As an application we will consider habit/routine formation as an inexact
proximal algorithm in the context of weak resistance to change. Because of
its strongly interdisciplinary aspect (Mathematics, Psychology, Economics,
Management), to be carefully justified, this application needs several
steps. This, because we need to compare three differents processes. First, a
general worthwhile change and stay process. Then, as two specific instances,
an inexact proximal algorithm and an habituation/routinization process. The
comparison must consider three aspects; a dynamical system (which one?),
which converges (how?), to an end point (which one?).

\begin{itemize}
\item \textbf{Step}.\textbf{1}. At the mathematical level, an exact proximal
algorithm represents a dynamical system, which, each step, minimizes a
proximal payoff $f+\lambda _{k}\left[ q(x^{k},\cdot )\right] ^{2}$ over the
whole space $X$. The perturbation term is $\lambda _{k}\left[ q(x^{k},y)%
\right] ^{2},\lambda _{k}>0$, $y\in X$, while the payoff function is $f$. An
inexact proximal algorithm represents a dynamical process which, each step,
uses a descent condition and a temporary stopping rule. In both cases, the
problem is to give conditions under which this process converges to a limit
point. Then, an exact or inexact proximal algorithm considers three points:

\begin{itemize}
\item[i)] a dynamical process. It represents the proximal sub-problem which
can be the minimization of the current proximal payoff (for an exact
proximal algorithm), or a descent condition, i.e., a sufficient decrease of
the proximal payoff (in the inexact case);

\item[ii)] the existence of some end points, which can be, or not, a
critical point, a local, or global minimum of $f$;

\item[iii)] and the convergence of the process towards an end point. This
convergence can be linear, quadratic\ldots.
\end{itemize}

\item \textbf{Step.2}. At the behavioral level, in the context of his
\textquotedblleft Variational rationality approach", (Soubeyran, 2009, 2010)
has examined \textquotedblleft worthwhile change and stay" processes. Such
dynamical processes refer to successions of \ temporary worthwhile changes
and worthwhile stays. The end points are variational traps. The convergence
of the process materializes in small steps, whose length goes to zero. Let
us remind that, in this behavioral context, end points represent traps
(reachable, i.e, more or less easy to reach, but difficult to leave).
Worthwhile changes balance, each step, motivation and resistance to change
forces. The motivation force is the utility $U\left[ A(x,y)\right] $ of the
advantages to change function $A(x,y)=f(x)-f(y)$ where $f$ represents an
unsatisfied need to be minimized. The resistance to change force represents
the disutility $D\left[ I(x,y)\right] $ of the inconvenients to become able
to change $I(x,y)=q(x,y)$, where $q(x,y)$ is a quasi distance$.$ In the
context of this paper, the variational approach considers the relative
resistance to change or aversion to change function $\Gamma \left[ q(x,y)%
\right] =U^{-1}\left[ D\left[ q(x,y)\right] \right] $. This shows that the
perturbation term of an exact or inexact proximal algorithm is a specific
instance of a relative resistance to change function $\Gamma \left[ q(x,y)%
\right] $. (Moreno et al. 2011) have considered the specific quadratic case
of weak resistance to change, where $\Gamma \left[ q(x,y)\right] =q(x,y)^{2}$%
. Then, at the behavioral level, two mains concepts, among others, drive a
\textquotedblleft worthwhile changes and stays" process: i) the unsatisfied
need function $f$ (which materializes the motivation to change) and, ii)
inertia (the relative resistance to change function $\Gamma \left[ q(x,y)%
\right] $)$.$

\item \textbf{Step}.\textbf{3}. Habit formation/routinization processes. A
very short survey has been given in the previous section Behavioral
stability issues: \textquotedblleft how habits and routines form and break".
These processes consider, i) a repetitive process where an action is
repeated again and again in the same recurrent context, ii) a final stage
where this action becomes a permanent habit, iii) the convergent process
which describes a slow learning habituation process where this action, being
repeated again and again in the same recurrent context becomes gradually
automatized.

\item \textbf{Step.4}. Then, it becomes clear that habituation/routinization
processes are specific instances of worthwhile change and stay processes.
Unsatisfied needs and inertia play a major role.

\item \textbf{Step.5.} In our paper we have shown how both\ exact and
inexact proximal algorithms are specific instances of \textquotedblleft
worthwhile change and stay" processes, because: i) minimization of the
current proximal payoff, descent conditions and current stopping rules are
special cases of worthwhile changes and marginal worthwhile stays, ii)
critical points, local and global minimum are specific representations of
variational traps, iii) convergence of the proximal algorithm shows how
proximal worthwhile changes converge, depending of the shape of the payoff
function (which can be convex, lower semicontinuous, or which can satisfy a
Kurdyka- Lojasiewicz inequality,\ldots) and of the shape of the
perturbation term (linear, convex,\ldots with respect to distance or quasi
distance).
\end{itemize}

\paragraph{ Why resistance to change matters much.}

The benchmark case of lower semicontinuous unsatisfied need functions $f$
and strong resistance to change functions $\Gamma \left[ q(x,y)\right] $
(where costs to be able to change are higher than a quasi distance) have
been examined by (Soubeyran, 2009, 2010) who considered worthwhile change
and stay processes. It has been shown that when a worthwhile change and stay
process converges to a variational trap, this variational formulation offers
a model of \ habit/routine formation which modelizes a permanent
habit/routine as the end point of a convergent path of worthwhile change and
temporary habits, where, for a moment, there is no way to do any other
worthwhile change, except repetitions.

The opposite case of weak resistance to change was left open. In the context
of an exact proximal algorithm, Moreno et al.(2011) have examined a specific
case of weak resistance to change, namely the quadratic case $\Gamma \left[
q(x,y)\right] =q(x,y)^{2}$. However, exact proximal algorithms represent a
very specific case of worthwhile changes, where, each step, the descent
condition is optimal. This means that, each step, the process minimizes the
proximal payoff $f+\lambda _{k}\Gamma \left[ q(x^{k},\cdot )\right] $ on the
whole space $X$. Then, such optimizing worthwhile changes are not step by
step economizing behaviors because they require to explore, each step, the
whole state space, again and again. The present paper considers the
generalized weak resistance to change case in the context of an inexact
proximal algorithm instead of an exact one. In both papers the unsatisfied
need function $f$ satisfies a Kurdyka- Lojasiewicz inequality. How the
strength of resistance to change impacts the speed of habit's/routine's
formation is, as an application, the topic of the related paper Bento,
Soubeyran (2014).

To summarize, we have compared an inexact generalized proximal algorithm
with a worthwhile change and stay process with respect to three aspects: A)
as a dynamical system, B) with an end point, C) which converges to that end
point. To end this paper, it remains to compare an inexact generalized
proximal algorithm with an habituation and routinization process, as it is
described in Psychology and Management Sciences, using the same three
criteria.

\paragraph{Inexact generalized proximal algorithm and habituation/
routinization process are dynamical systems.}

\begin{itemize}
\item \textbf{Habituation/routinization process}. They represents the
repetition of an action of a given kind (some activity related to a given
goal), in order to satisfy a recurrent unsatisfied need in a stable context.
The repetition concerns the action and what becomes more and more the same
is \textquotedblleft the way of doing it" (the script). This repetition
follows a succession of worthwhile changes and stays. An inexact proximal
algorithm represents a step by step processes, a succession of moves in
order to "decrease enough" some proximal payoff function. Usually, both
dynamical processes are unable to reach the goal in one step. Each step, the
level of satisfaction of the recurrent need increases, but some
unsatisfaction remains.

An habituation process is driven by two balancing forces : a motivation to
change function $M\left[ A(x,y)\right] $ (an habit/routine must serve us),
and a resistance to change function $D\left[ I(x,y)\right] $,
(habits/routines are hard to form and hard to break because learning and
unlearning are costly). An inexact proximal algorithm is driven by the two
terms $f(y)$ and $\Gamma \left[ q(x,y)\right] $ of its proximal payoff $%
f(y)+\lambda _{k}\Gamma \left[ q(x,y)\right] $ where $q(x,y)=C(x,y)$. This
balance describes the goal-habit interface.

The rationality of an habituation process is to improve by repetition the
way of doing a similar action in the same context. The agent improves with
costs to change. He satisfices, doing worthwhile changes, without exploring
too much each step (local consideration and exploration; see Soubeyran,
(2009, 2010) for this important aspect). An inexact proximal algorithm
follows, each step, some descent condition and marginal stopping rule,
without optimizing each step.

The influence of the past differs from one process to the other. For an
habituation process the impact of the past can be very important (the past
sequence matters much). For an inexact proximal algorithm it is as if only
the last action matters. The influence of the past is minimal (it is as if
the agent has a short memory). The influence of the future seems identical
in both cases: myopia seems to be the rule. Only the next future action
matters. Agent's behavior driven by habits/routines are not forward looking.
For more forward looking worthwhile to change behavior; see Soubeyran,
(2009, 2010).

\textbf{Convergence: see the last paragraph (``why resistance to change
matters much")}

\textbf{End points. }Our inexact proximal algorithm converges to a critical
point, which is not an end point, unless it can be shown that a critical
point is a variational trap (as this is done in our paper). A variational
trap is worthwhile to reach and not worthwhile to leave. An habituation
process ends in a permanent habit/routine which is hard to form and hard to
break. It represents the vestige of a past repeated behavior.
\end{itemize}

\section{Making the assumptions of the proximal algorithms clear in
behavioral terms}

We have to show how, in Behavioral Sciences, our three proximal algorithms
modelize, at least in a reduced form, how habits form and break.

A first step has been given before, in Section 7. This have been done in
five steps. We have shown that inexact proximal algorithms and habitual
processes are dynamical systems, we have compared their end points, critical
points or variational traps, we have examined how the strength of resistance
to change influences their abilities to converge, and we have linked
resistance to change to loss aversion, a famous behavioral concept.

A second step is to detail the behavioral content of all the hypothesis
which drive our three proximal algorithms. There are general behavioral
hypothesis which are common to the three proximal algorithms, and specific
hypothesis relative to each of them. More explicitly, the three algorithms
suppose that costs to be able to change are quasi distances. They consider,
each step, worthwhile and marginally worthwhile changes (descent
conditions), and suppose weak resistance to change, as well as a marginal
stopping rule. The first algorithm is targeted to converge to a critical
point. The second and third algorithms are targeted to converge to a
variational trap.

\paragraph{General behavioral hypothesis}

\begin{itemize}
\item H.1. Costs to be able to change $C$ are modelized as quasi-distances.
For an agent, costs to be able to change $C(x,y)=q(x,y)$ refer to the
infimum of the costs to be able to change his capabilities, from having the
capability to do an action $x,$ to the capability to do an action $y$. Then $%
C(x,x)=0$ means that if the agent is able to do an action $x,$he is able to
do this action at no cost. The condition $C(x,y)=0\Longleftrightarrow y=x$
means that if the agent is able to move at no cost, then, he can only repeat
the same action, if he is able to do it . The triangle inequality $%
C(x,z)\leq C(x,y)+C(y,z)$ for all $x,y,z\in X$ means that, for an agent, the
infimum cost to change his capabilities from the initial capability to do an
action $x$ to the final capability to do an action $z$ is lower than the
infimum cost to change his capabilities from the initial capability to do
action $x$ to the intermediate capability to do an action $y$ and the
infimum cost to change his capabilities from the intermediate capability to
do action $y$ to the final capability to do action $z$, because the way to
change successively from an initial capability to an intermediate capability
and from this intermediate capability to a final capability is an indirect
way to change from the initial to the final capability.

\item H.2. Unsatisfied needs $f:\mathbb{R}^{n}\rightarrow \mathbb{R}\cup
\{+\infty \}$ is a proper lower semicontinuous function. This is a
regularity assumption which supposes no free lunch. The agent cannot reduce
his unsatisfied need in a given small amount without changing enough his
action (no jump downward are allowed).

\item H.3. Advantages to change $A(x,y)=f(x)-f(y)${\huge \ }refer to
separable advantages to change functions, linked, when positive, to a
decrease in unsatisfied needs.

\item H.4. The ratio $\xi _{k}>0$ modelizes how much a change can be
worthwhile (the adaptive satisficing case). This ratio can change from
period to period. This means that the agent can adapt with delay or not,
each step, his degree of satisficing. In this paper he adapts with one
period delay (writing $\xi _{k}>0$ instead of $\xi _{k+1}>0,$ to fit with a
proximal formulation).

\item H.5. The utility function $U[\cdot ]$ is invertible with $U[0]=0$.
This is the case for a strictly increasing utility function, relative to
advantages to change (the usual case).

\item Condition (4) supposes that $\Gamma (C)=U^{-1}\left[ D%
\left[ C\right] \right] $ is twice differentiable with respect to $C.$ It is
a regularity condition, relative to the relative resistance to change
function.

\item Condition (7)  means that the relative resistance to
change function is \textquotedblleft flat enough in the small" (in the "weak
enough resistance to change" case). It supposes that the margnal relative
resistance to change must be \textquotedblleft lower enough" with respect to
the mean relative resistance to change, at least for \textquotedblleft low
enough" costs to be able to change.

\item Assumption 3.1 supposes that costs to be able to change
are high (low) enough iff the old and new actions are different (similar)
enough.
\end{itemize}

\paragraph{Hypothesis relative to Algorithm 4.1}

Algorithm 4.1 supposes three conditions  (10), (11), (12).
Condition (10) imposes worthwhile changes (a sufficient descent assumption)
along the process. Condition (11) defines subgradients for the unsatisfied
needs and costs to be able to change functions. Condition (12) refers, each
step, to a stopping rule, where the norm of the marginal decrease of the
unsatisfied need is lower than the norm of the marginal relative resistance
to change.


The Kurdyka-Lojasiewicz inequality (Definition 4.3)  refers to a
curvature property of the unsatisfied need function, near a critical point.
The unsatisfied need must be lower than some increasing function of the
marginal unsatisfied need, close to a critical point.

Assumption 3.2 supposes that the marginal relative resistance to change
function is lower than a power function, near the origin. It is a curvature
property.

\paragraph{Hypothesis relative to Algorithm (9)}

This algorithm supposes approximate (almost) worthwhile changes,
including, at each step, an error term $\varepsilon_{k}>0$.

\paragraph{Hypothesis relative to Algorithm 4.2}

This algorithm supposes three conditions (13), (14), (15). Condition (13) is
a modified worthwhile to change assumption. Condition (14) defines
subgradients for the unsatisfied needs and costs to be able to change
functions. Condition (15) is the stopping rule condition (12). This
algorithm adopts all the hypothesis of Algorithm 4.1.

\section{Application to the Formation and Break of Habits/Routines}

The {\it Variational rationality (VR) approach}, see Soubeyran (2009, 2010), focus attention on interdisciplinary stability and change dynamics (habits
and routines, creation and innovation, exploration and exploitation,\ldots),
and the self regulation problem, seen as a stop and go course pursuit
between feasible means and desirable ends mixing, in alternation, discrepancy
production (goal setting, goal revision), discrepancy reduction (goal
striving, goal pursuit) and goal disengagement. It rests on two main
concepts, worthwhile temporary stays and changes, variational traps and nine
principles. This (VR) approach allows to recover the main mathematical
variational principles, and in turn, it benefits from almost all variational
algorithms for procedural applications, which all, use some of the main
variational rationality principles.
%
The {\it Habitual domain (HD) theory and (DMCS)
approach} (see Yu, Chen (2010), for a nice presentation) and the (DMOCS) Decision
making and optimization problems in changeable spaces (see Larbani, Yu (2012)) refer to three stability and change problems. They are i) stability
issues, using a system of differential equations, a variant of the famous
pattern formation Cohen-Grossberg model (see Cohen, Grossberg (1983)), ii)
expansion of an initial competency set to be able to solve a given problem,
which requires to acquire a new given competency set, using mathematical
programming methods and graphs, iii) DMCS optimization and game problems,
using Markov chains, with applications to innovation cover-discover problems
(see Yu, Larbani (2009) and Larbani, Yu (2009,2011)).

\subsection*{\textbf{1) Habit/routine formation and break problems}}

The (VR) approach sees habit/routine formation and break as a balance between motivation
and resistance to change, when agents use worthwhile changes; Permanent
habits refer to variational traps, as the end of a succession of worthwhile
changes. These findings fit well with Psychology and Management theories of
habits and routines formation and break. For example, in Psychology, habits
form by repetitions, in a rather stable context, which trigger their
repetition. They become gradually more and more automatized behaviors which
are intentional (goal directed), more or less conscious and controllable,
and economize cognitive resources. In this context agents are bounded
rational. The (HD) approach modelizes habit formation as a balance, each
step, between excitation and inhibition forces, which determine, each time,
the variation of the allocation of attention and effort (propensities),
using, as said before, a variant of the Cohen-Grossberg system of
differential equations (see Cohen, Grossberg (1983)).

In the limit, the allocation of attentions and efforts (propensities)
converge to an allocation which represents a stable habitual domain. The
(DMOCS) approach of games defines limit profiles of mind sets as absorbing
states of a Markow chain. In these two contexts agents are bounded rational.

\subsection*{\textbf{2) Inexact proximal algorithms as repeated
satisficing problems with changeable decision sets}}

We have shown in Section 3 of this paper that
our inexact proximal algorithms refer to variable and changeable decision
sets, payoffs, goals (satisficing worthwhile changes) and preference
processes. This is the case because worthwhile to change sets are changeable decision sets which change, each period, with experience and the choice, each period, of the satisficing worthwhile to change ratio; see the whole Section 2 and the dynamic of worthwhile (hence satisficing) temporary stays and changes. More precisely Inexact generalized proximal algorithms are
specific instance of VR worthwhile stay and change adaptive dynamics. They
consider non transitive variable worthwhile to change preferences, see Section 2, which are
reference dependent preferences with variable reference points (variable
experience dependent utility/disutility functions, variable payoff
functions, variable and non linear resistance to change functions  via, in each case, of the introduction of the separable and variable lambda term). They can
use costs to be able to change which do not satisfy the triangular
inequality; see Bento et al., (2014) and several other references within Bento, Soubeyran (2014). They are inexact and
procedural algorithms on quasi metric spaces. Then, they modelize bounded
rational and more or less myopic agents who bracket difficult decisions in
several steps, contrary to exact proximal algorithms which modelize a
repeated optimization problem, and which are not the topic of our paper.
They generalize the Simon satisficing principle to a dynamical context
(repeated and adaptive satisficing). They are anchored to optimization
processes as benchmark cases. They deal with changeable spaces as changeable worthwhile to change sets and also changeable exploration sets; see
Attouch, Soubeyran, (2011), Bento, Cruz Neto, Soubyeran (2014) and consider
convergence in finite time as a central topic; see Bento, Soubeyran (2014) and Bento et al. (2013). They include psychological aspects (like motivations, cognitions
and inertia) and can easily include emotional aspects.

We point out that the assumptions of inexact proximal
algorithms explicit in behavioral terms have been given in Section 2 of this paper (see the example) and after
each proximal algorithm.



\section{References on the VR, HD and DMCS approaches}

\subsection*{References for the Habitual domain theory}

\noindent 1. Yu, P.L, Chen,Y.C.: Dynamic MCDM, habitual domains and
competence set analysis for effective decision making in changeable spaces.
Chaper 1. In Trends in Multiple Criteria Decision Analysis. Matthias Ehrgott
Jose Rui Figueira Salvatore Greco. Springer  (2010)

\vspace{.5cm}

\noindent 2. Shi, D. S., Yu,  P. L.: Optimal expansion of competency sets with
intermediate skills and compound nodes. Journal of Optimization Theory and
Applications, Vol. 102, No.1, pp. 643-657, 1999.

\vspace{.5cm}

\noindent 3. Yu, P. L., Zhang, D.: A Foundation for Competence Set
Analysis, Mathematical Social Sciences, Vol. 20, pp. 251-299, 1990.

\vspace{.5cm}

\noindent 4. Yu, P. L., Zhang, D.: A Marginal Analysis for Competence
Set Expansion, Journal of Optimization Theory and Applications, Vol. 76, pp.
87-109, 1993.

\vspace{.5cm}

\noindent 5. Li, H. L., Yu, P. L.: Optimal Competence Set Expansion
Using Deduction Graphs, Journal of Optimization Theory and Applications,
Vol. 80, pp. 75-91, 1994.

\vspace{.5cm}

\noindent 6. Shi, D. S.,  Yu,  P. L.: Optimal expansion and design of
competence set with asymmetric acquiring costs. Journal of Optimal Theory
and Applications, Vol. 88, pp. 643--658, 1996.

\vspace{.5cm}

\noindent 7. Li, J. M., Chiang, C. I., Yu, P. L.: Optimal multiple
stage expansion of competence set. European Journal of Operational Research,
Vol. 120(3), pp. 511-524, 2000.

\vspace{.5cm}

\noindent 8. Yu, P. L.: Habitual Domains. Operations Research, Vol. 39 (6),
869--876, 1991.

\noindent 9. Chan, S. J., Yu, P.L.: Stable Habitual Domains: existence and
implications. Journal of Mathematical Analysis and Applications, Vol. 110 (2)
469-482, 1985.

\vspace{.5cm}

\noindent 10. Yu, P. L., Larbani M.: Two-person second-order games, Part 1:
formulation and transition anatomy. Journal of Optimization Theory and
Applications, 141(3), 619-639, 2009.

\vspace{.5cm}

\noindent 11. Larbani, M., Yu, P.L.: Two-person second-order games, Part II:
restructuring operations to reach a win-win profile, Journal of Optimization
Theory and Application, Vol. 141, 641-659, 2009.

\vspace{.5cm}

\noindent 12. Larbani, M., Yu, P.L.: n-Person second-order games: A paradigm
shift in game theory, Journal of Optimization Theory and Application, Vol. 149
(3), 447--473, 2011.

\vspace{.5cm}

\noindent 13. Larbani M., Yu P.L. Decision Making and Optimization in
Changeable Spaces, a New Paradigm, Journal of Optimization Theory and
Application, Vol. 155 (3), 727-761, 2012.

\vspace{.5cm}

\noindent 14. Po L. Yu, Forming Winning Strategies, An Integrated Theory of
Habitual Domains, Springer-Verlag, Berlin, Heidelberg, New York, London,
Paris, Tokyo, 1990 (392 pages).
\vspace{.5cm}

\noindent 15. Po-Lung Yu, Habitual Domains: Freeing Yourself from the Limits on
Your Life, Highwater Editions, Kansas City, July 1995 (224 pages). (also
available in Chinese and Korean translation)
\vspace{.5cm}

\noindent 16. Po-Lung Yu, Habitual Domains and Forming Winning Strategies, NCTU
(National Chiao Tung University) Press, Hsinchu, Taiwan, 2002 (531 pages).
\vspace{.5cm}

\subsection*{References for the Variational rationality approach}

\subsection*{Publications}

\noindent 1) Attouch, H. Soubeyran, A.: (2006)  Inertia and reactivity in
decision making as cognitive variational inequalities.
Journal of Convex Analysis, 13 (2), 207--224

\vspace{.5cm}

\noindent 2) Martinez Legaz, J.E., Soubeyran, A.: (2007) A Tabu search scheme for
abstract problems, with applications to the computation of fixed points
Journal of Mathematical Analysis and Applications, 338 (1), 620-627

\vspace{.5cm}

\noindent 3) Atouch, H., Redont, P., Soubeyran, A.: (2007) A New class of alternating
proximal minimization algorithms with costs- to- move. SIAM Journal on
Optimization, 18 (3), 1061-1081

\vspace{.5cm}

\noindent 4) Attouch, H., Bolte, J., Redont, P., Soubeyran, A.: (2008) Alternating proximal
algorithms for weakly coupled convex minimization problems. Applications to
dynamical games and PDE's. Journal of Convex Analysis,15
(3), 485-506

\vspace{.5cm}

\noindent 5) Cruz Neto, J. X., Oliveira, P. R., Souza, S. S., Soubeyran, A.: (2010 ) A Proximal
point algorithm with separable Bregman distances for quasiconvex
optimization over the nonnegative orthant. European Journal of Operation
Research, (201) (2), 365-376

\vspace{.5cm}

\noindent 6) Luc, D. T., Sarabi, E.,  Soubeyran, A.: ( 2010) Existence of solutions in
variational relations problems without convexity. Journal of Mathematical
Analysis and Applications,  364 (2), 544-555

\vspace{.5cm}

\noindent 7) Attouch, H., Bolte, J.,  Redont, P., Soubeyran, A.: (2010) Proximal
Alternating Minimization and Projection Methods for Nonconvex Problems. An
Approach Based on the Kurdyka- Lojasiewicz Inequality, Mathematics of
Operations Research, 35 (2), 438-457

\vspace{.5cm}

\noindent 8) Attouch, H., Soubeyran, A.: (2010) Local search proximal algorithms as
decision dynamicswith costs to move, Set-Valued and Variational Analysis, 19 (1), 157-177

\vspace{.5cm}

\noindent 9) Moreno, F., Oliveira, P. R., Soubeyran, A.: (2012) A proximal algorithm with
quasi distance. Application to habits formation, Optimization, 61 (12), 1383-1403

\vspace{.5cm}

\noindent 10) Flores Bazan, F., Luc, T., Soubeyran, A.: (2012) Maximal
elements under reference-dependent preferences with applications to
behavioral traps and games. Journal of Optimization
Theory and Applications, 155 (3), 883-901.

\vspace{.5cm}

\noindent 11) Godal, O., Flam, S., Soubeyran, A.: ( 2012) Gradient
differences and bilateral barters,1-20, I First  http://dx.doi.org/10.1080/02331934.2012.679940

\vspace{.5cm}

\noindent 12) Luc, D. T., Soubeyran, A.: (2013) Variable preference relations:
Existence of maximal elements. Journal of Mathematical Economics, 49(4),
251-262

\vspace{.5cm}

\noindent 13) Cruz Neto, J. X., Oliveira, P. R., Soares Jr, P. A., Soubeyran, A.
(2013) Learning how to play Nash, potential games and alternating
minimization method for structured non convex problems on Riemannian
manifolds. Journal of Convex Analysis, 20(2), 395-438

\vspace{.5cm}

\noindent 14) Cruz Neto, J. X., Oliveira, P. R., Soares Jr, P. A., Soubeyran, A.: (2013)
Proximal Point Method on Finslerian Manifolds and the \textquotedblleft
Effort--Accuracy\textquotedblright\ Trade-off. Journal of Optimization
Theory and Applications, 1-19. \\
 http://dx.doi.org/10.1007/s10957-013-0483-5

\vspace{.5cm}

\noindent 15) Villacorta.K, Oliveira.P, Soubeyran.A. (2013). A Trust-Region Method for
Unconstrained Multiobjective Problems with Applications in Satisficing
Processes. Journal of Optimization Theory and Applications, 160 (3),  865-889

\vspace{.5cm}

\noindent 16) Bento, G. C., Cruz Neto, J. X., Oliveira, P. R., Soubeyran, A.:
(2014). The self regulation problem as an inexact steepest descent method
for multicriteria optimization. European Journal of Operational Research.
235, 494-502

\vspace{.5cm}

\noindent 17) Long, N., Soubeyran, A., \& Soubeyran, R.(2013). Knowledge Accumulation
within an Organization. Accepted for publication in the International
Economic Review.
\vspace{.5cm}

\noindent 18) Bento, G., Cruz Neto, J. X., Soubeyran, A.: (2014) A Proximal Point-Type Method
for Multicriteria Optimization. Set-Valued and Variational Analysis: theory and applications, v. 22, p. 557-573, 2014.

\section*{In revision}

\noindent 19) Bento, G. C., Soubeyran, A.: ( 2014) Generalized Inexact
Proximal Algorithms: Habit's Formation with Resistance to Change, following
Worthwhile Changes.  JOTA.

\vspace{.5cm}

\noindent 20) Bento, G. C., Cruz Neto, J. X.,  Soares, P., Soubeyran, A.: (2014) Variational Rationality and the Equilibrium Problem on Hadamard
Manifolds. JOTA

\vspace{.5cm}

\section*{Submitted papers}

\noindent 21) Martinez Legaz, J. E., Soubeyran, A.: (2013) Convergence
in sequential decision making with learning and costs to change, soumis
Systems and Control Letters.

\vspace{.5cm}

\noindent 22) Bao, T., Mordukhovich, B., Soubeyran, A.: (2013) Variational Analysis in Psychological Modelling.  JOTA.

\vspace{.5cm}

\noindent 23) Bento, G. C.,  Cruz Neto, J. X., Soares, P.,  Soubeyran, A.: (2014) Bilevel equilibrium problems as limits of variational
traps.  JOTA

\vspace{.5cm}

\noindent 24) Bao, T., Mordukhovich, B., Soubeyran, A.: (2013) Variational principles, variational rationality and the Sen capability
approach of well being.  EJOR.

\vspace{.5cm}

\noindent 25) Moreno, F., Oliveira, P. R., Soubeyran, A.:(2013) Dual
equilibrium problems: how a succession of aspiration points converges to an
equilibrium. To be submitted.

\section*{ARXIV papers}

\noindent 26) Attouch, H., Soubeyran, A.: (2009) Worthwhile-to-move
behaviors as temporary satisficing without too many sacrificing
processes. 

\vspace{.5cm}

\noindent 27) Bao, T., Mordukhovich, B., Soubeyran, A.: (2013) Variational principles in
models of Behavioral Sciences. 

\vspace{.5cm}

\noindent 28) Bento, G. C., Soubeyran, A.: (2013) Some Comparisons between
Variational rationality, habitual Domain and DMCS
approaches.

\section*{Pre prints}

\noindent 29) Soubeyran, A.: (2009) Variational rationality, a theory of individual
stability and change: worthwhile and ambidextry behaviors.

\vspace{.5cm}

\noindent 30) Soubeyran, A.: (2010) Variational rationality and the unsatisfied man:
routines and the course pursuit between aspirations, capabilities and
beliefs.

\subsection*{Other references}

\noindent 31) Tversky, A., Kahneman, D.: Loss aversion in riskless choice: a
reference dependent model. Q. J. Econ. \textbf{106}(4), 1039-1061 (1991)

\vspace{.5cm}

\noindent 32) Simon. H.: A behavioral model of rational choice. Q. J. Econ. 
\textbf{69}, 99-118 (1955)

\vspace{.5cm}

\noindent 33) Kahneman, D., Tversky, A.: Prospect theory: An analysis of
decision under risk. Econometrica \textbf{47}(2), 263-291 (1979)

\vspace{.5cm}

\noindent 34) Abdellaoui, M., Bleichrodt, H., Paraschiv, C.: Loss aversion
under prospect theory: A parameter-free measurement. Manage. Sci. \textbf{53}%
(10), 1659-1674 (2007)

\vspace{.5cm}

\noindent 35) K\"{o}bberling, V., Wakker, P.: An index of loss aversion. J.
Econ. Theory. \textbf{122},119--131 (2005)

\vspace{.5cm}

\noindent 36) Verplanken, B., Aarts, H.: Habit, attitude and planned
behaviour: Is habit an empty construct or an interesting case of
goal-directed automaticity? Eur. Rev. Soc. Psychol. \textbf{10}, 101--134
(1999)

\vspace{.5cm}

\noindent 37) Duhigg, C.: The power of habits. Why we do what we do in life
and business. Random House, New York (2012)

\vspace{.5cm}

\noindent 38) Costa, A., Kallick, B.: Habits of mind. A developmental
deries. Association for Supervision and Curriculum Development, Alexandria,
VA (2000)

\vspace{.5cm}

\noindent 39)  Gardner, B.: Habit as automaticity, not frequency. Eur. Health
Psychol. \textbf{14} (2) 32-36 (2012)

\vspace{.5cm}

\noindent 40) Verplanken, B., Wood, W.: Interventions to break and create
consumer habits. J. Public Policy Mark. \textbf{25} (1), 90 - 103 (2006)

\vspace{.5cm}

\noindent 41) Bargh, J. A.: The four horsemen of automaticity: Awareness,
intention, efficiency and control in social cognition. In R.S. Wyer and T.K.
Srull (Eds), Handbook of Social Cognition. Hillsdale, N. J. Erlbaum (1994).

\vspace{.5cm}

\noindent 42) Aarts, H., Dijksterhuis, A.: The automatic activation of goal
directed behaviour: The case of travel habit. J. Environ. Psychol. \textbf{20%
}, 75-82 (2000)

\vspace{.5cm}

\noindent 43) Verplanken, B., Aarts, H.: Habit, attitude and planned
behaviour: Is habit an empty construct or an interesting case of
goal-directed automaticity? Eur. Rev. Soc. Psychol. \textbf{10}, 101 -- 134
(1999)

\vspace{.5cm}

\noindent 44) Lally, P., Van Jaarsveld, C., Potts, H., Wardle, J.: How are
habits formed: Modelling habit formation in the real world. Eur. J. Soc.
Psychol. \textbf{40}(6), 998-1009 (2010)

\vspace{.5cm}

\noindent 45) Lewin, K.: Field theory in social sciences. NY: Harper \& Row,
New York (1951)

\vspace{.5cm}

\noindent 46) Rumelt, R.: Inertia and transformation. Montgomery, Cynthia
A., ed., Resources in an Evolutionary Perspective: Towards a Synthesis of
Evolutionary and Resource-Based Approaches to Strategy, Kluwer Academic
Publishers, Norwell, Mass. 101-132 (1995)

\vspace{.5cm}

\noindent 47) Li, B.: The reviews and prospects of studies on habits.
Psychol. Sci. \textbf{35}(3) 745-753 ( 2012)

\vspace{.5cm}

\noindent 48) Ajzen. I.: The theory of planned behavior. Organ. Behav. Hum.
Dec. \textbf{50}, 179-211 (1991)

\vspace{.5cm}

\noindent 49) Verplanken, B., Orbell, S.: Reflections on past behavior: a
self-report index of habit strength. J. Appl. Soc. Psychol. \textbf{33}(6),
1313--1330 (2003)

\vspace{.5cm}

\noindent 50) Wood, W., Neal, D.: A new look at habits and the habit-goal
interface. Psychol. Rev. \textbf{114}(4), 843-863 ( 2007)

\vspace{.5cm}

\noindent 51) Becker.G, Murphy, K.: A theory of rational addiction. J.
Polit. Econ. \textbf{96}, 675-700 (1988)

\vspace{.5cm}

\noindent 52) Abel, A.: Asset prices under habit formation and catching up
with the Joneses. Am. Econ. Rev. \textbf{80}(2), 38-42 (1990)

\vspace{.5cm}

\noindent 53) Carroll, C.: Solving consumption models with multiplicative
habits. Econ. Lett. \textbf{68}, 67-77 (2000)

\vspace{.5cm}

\noindent 54) Wendner, R.: Do habits raise consumption growth? Research in
Economics \textbf{57}, 151--163 (2003)

\vspace{.5cm}

\noindent 55) Becker, M.: Organizational routines: a review of the
literature. Oxford J. Econ. \& Soc. Sci.- Industrial and Corp. Ch. \textbf{13%
}(4), 643-678 (2004)

\vspace{.5cm}

\noindent 56) Cohen, M., Grossberg, S.: Absolute stability of global
pattern formation and parallel memory storage by competitive neural
networks. IEEE Transactions on Systems, Man, and Cybernetics, SMC, {\bf 13}, 815-826 (1983)

\vspace{.5cm}

\noindent 57) Grossberg, S.: Contour enhancement, short term memory, and
constancies in reverberating neural network, Stud. Appl. Math. {\bf 3} (3)
213-256, (1973)

\vspace{.5cm}

\noindent 58) Grossberg, S.: Competition, decision, and consensus, J.
Math. Anal. 66, 47-93, (1978)

\vspace{.5cm}

\noindent 59) Grossberg, S.: How does a brain build a cognitive code?
Psyhal. Rer. {\bf 87} (1), 1-51, (1980) 

\vspace{.5cm}

\noindent 60)  Kunreuther, H., Meyer, R., Zeckhauser, R., Slovic, P.,
Schwartz, B., Schade, C., Hogarth, R.: High stakes decision
making: Normative, descriptive and prescriptive considerations. Marketing
Letters, {\bf 13} (3), 259-268 (1980)

\end{document}